\documentclass[12pt,reqno]{amsart}
\usepackage{amsmath}
\usepackage{eucal}
\usepackage{cite}
\usepackage{graphicx}

\newtheorem{defi}{Definition}
\newcommand{\brdef}{\begin{defi}}
\newcommand{\erdef}{\end{defi}}

\newtheorem{cor}{Corollary}
\newcommand{\bcor}{\begin{cor}}
\newcommand{\ecor}{\end{cor}}

\newtheorem{thm}{Theorem}
\newcommand{\bth}{\begin{thm}}
\newcommand{\eth}{\end{thm}}
\newtheorem{lem}{Lemma}
\newcommand{\ble}{\begin{lem}}
\newcommand{\ele}{\end{lem}}

\def\pn{\par\noindent}

 \huge

\usepackage{hyperref}

\numberwithin{equation}{section}

\begin{document}
\begin{center}
{\large \textbf{Results on para-Sasakian manifold admitting a quarter symmetric metric connection}} \\
\
\\
Vishnuvardhana. S.V., Venkatesha
\end{center}

\begin{quotation}
{\bf Abstract:} In this paper we have studied pseudosymmetric, Ricci-pseudosymmetric and projectively pseudosymmetric para-Sasakian manifold admitting a quarter-symmetric metric connection and constructed examples of 3-dimensional and 5-dimensional para-Sasakian manifold admitting a quarter-symmetric metric connection to verify our results.
\end{quotation}
\section{Introduction\label{sec1}}
\par One of the most important geometric property of a space is symmetry. Spaces admitting some sense of symmetry play an important role in differential geometry and general relativity. Cartan \cite{ECartan} introduced locally symmetric spaces, i.e., the Riemannian manifold $(M, g)$ for which $\nabla R=0$, where $\nabla$ denotes the Levi-Civita connection of the metric. The integrability condition of $\nabla R=0$ is $R\cdot R=0$. Thus, every locally symmetric space satisfies $R\cdot R=0$, whereby the first R stands for the curvature operator of $(M, g)$, i.e., for tangent vector fields $X$ and $Y$ one has $R(X, Y)= \nabla_{X}\nabla_{Y} - \nabla_{Y}\nabla_{X} -\nabla_{[X, Y]}$, which acts as a derivation on the second $R$ which stands for the Riemann-Christoffel curvature tensor. The converse however does not hold in general. The
spaces for which $R\cdot R=0$ holds at every point were called semi-symmetric spaces and which were classified by Szabo \cite{ZISzabo}.

\par Semisymmetric manifolds form a subclass of the class of pseudosymmetric manifolds. In some spaces $R\cdot R$ is not identically zero, these turn
out to be the pseudo-symmetric spaces of Deszcz \cite{RDeszcz, RDeszczSYaprak, RDMGMHGZ}, which are characterised by the condition $R\cdot R = L \,Q(g, R)$, where $L$ is a real function on $M$ and $Q(g, R)$ is the Tachibana tensor of $M$.
\par If at every point of $M$ the curvature tensor satisfies the condition
\begin{equation}
 R(X, Y)\cdot \CMcal{J}=L_{\CMcal{J}}[(X\wedge_{g} Y)\cdot \CMcal{J}], \label{I1}
\end{equation}
then a Riemannian manifold $M$ is called pseudosymmetric (resp., Ricci-pseudosymmetric, projectively pseudosymmetric) when $\CMcal{J}=R (resp., S, P)$ . Here $(X\wedge_{g} Y)$ is an endomorphism and is defined by $(X\wedge_{g} Y)Z=g(Y, Z)X-g(X, Z)Y$ and $L_{\CMcal{J}}$ is some function on $U_{\CMcal{J}} = \{x \in M : \CMcal{J} \neq 0\}$ at $x$. A geometric interpretation of
the notion of pseudosymmetry is given in \cite{SHLV}. It is also easy to see that every
pseudosymmetric manifold is Ricci-pseudosymmetric, but the converse is not true.

\par An analogue to the almost contact structure, the notion of almost paracontact structure was introduced by Sato \cite{ISato}. An almost contact manifold is always odd-dimensional but an almost paracontact manifold could be of even dimension as well. Kaneyuki and Williams \cite{SKFLW} studied the almost paracontact structure on a pseudo-Riemannian manifold. Recently, almost paracontact geometry in particular, para-Sasakian geometry has taking interest, because of its interplay with the theory of para-Kahler manifolds and its role in pseudo-Riemannian geometry and mathematical physics (\cite{DVAVCASGTL, VCCMTMFS, VCMALLS}, etc.,).

\par As a generalization of semi-symmetric connection, quarter-symmetric connection was introduced. Quarter-symmetric connection on a differentiable manifold with affine connection was defined and studied by Golab \cite{SGolab}. From thereafter many geometers studied this connection on different manifolds.

\par Para-Sasakian manifold with respect to quarter-symmetric metric connection was studied by De et.al., [\cite{KMandalUCDe, AKMUCD}], Pradeep Kumar et.al., \cite{KTPKVCSB} and Bisht and Shanker \cite{LBSS}.

\par Motivated by the above studies in this article we study properties of projective curvature tensor on para-Sasakian manifold admitting a quarter-symmetric metric connection. The organization of the paper is as follows: In Section \ref{sec2}, we present some basic notions of para-Sasakian manifold and quarter-symmetric metric connection on it. Section \ref{sec3} and \ref{sec4} are respectively devoted to study the pseudosymmetric and Ricci-pseudosymmetric para-Sasakian manifold admitting a quarter-symmetric metric connection. Here we prove that if a para-Sasakian manifold $M^n$ admitting a quarter-symmetric metric connection is Pseudosymmetric (resp., Ricci pseudosymmetric) then $M^n$ is an Einstein manifold with respect to quarter-symmetric metric connection or it satisfies $L_{\tilde{R}}=-2$ (resp., $L_{\tilde{S}}=-2$). Section \ref{sec5} and \ref{sec6} are concerned with projectively flat and projectively pseudosymmetric para-Sasakian manifold $M^n$ admitting a quarter-symmetric metric connection. Finally, we construct examples of 3-dimensional and 5-dimensional para-Sasakian manifold admitting a quarter-symmetric metric connection and we find some of its geometric characteristics.

\section{Preliminaries\label{sec2}}
 A differential manifold $M^{n}$ is said to admit an almost paracontact Riemannian structure $(\phi, \xi, \eta, g)$, where $\phi$ is a tensor field of type (1, 1), $\xi$ is a vector field, $\eta$ is a 1-form and $g$ is a Riemannian metric on $M^{n}$ such that
\begin{eqnarray}
\label{1.1} \phi^2= I-\eta \,o\,\xi,\,\,\,\,\,\,\,\, \eta(\xi)=1,\,\,\,\,\,\,\, \phi(\xi)=0,\,\,\,\,\,\,\,\,\, \eta(\phi X)=0,\\
\label{1.2} g(X, \xi)= \eta(X), \,\,\,\,\,\,\,\, g(\phi X, \phi Y)=g(X,Y)-\eta(X)\eta(Y),
\end{eqnarray}
for all vector fields $X, Y\in T(M^{n})$. If $(\phi, \xi, \eta, g)$ on $M^n$ satisfies the following equations
\begin{eqnarray}
\label{1.3} (\nabla_{X}\phi)Y= - g(X, Y)\xi - \eta(Y)X + 2\eta(X)\eta(Y)\xi,\\
\label{1.4} d\eta=0 \,\,\,\,\,\,and \,\,\,\,\,\, \nabla_{X}\xi=\phi X,
\end{eqnarray}
then $M^{n}$ is called para-Sasakian manifold \cite{ABARMAN}.

In a para-Sasakian manifold, the following relations hold \cite{CihanOzgur}:
\begin{eqnarray}
\label{1.5} (\nabla_{X}\eta)Y= - g(X, Y) + \eta(X)\eta(Y),
\end{eqnarray}
\begin{eqnarray}
\label{1.6} \eta(R(X, Y)Z)= g(X, Z)\eta(Y)-g(Y, Z)\eta(X),\\
\label{1.8} R(X, Y)\xi= \eta(X)Y-\eta(Y)X, \,\,\,\,\, R(\xi, X)Y= \eta(Y)X - g(X, Y)\xi,\\
\label{1.9} S(X, \xi)=-(n-1) \eta(X),\\
\label{1.10} S(\phi X, \phi Y)=  S(X, Y) + (n-1)\eta(X)\eta(Y),
\end{eqnarray}
for every vector fields $X, Y, Z$ on $M^n$.

\par Here we consider a quarter-symmetric metric connection $\tilde{\nabla}$ on a para-Sasakian
manifold \cite{KMandalUCDe} given by
\begin{equation}
\tilde{\nabla}_{X}Y = \nabla_{X}Y + \eta(Y) \phi X - g(\phi X, Y)\xi \label{1.11}.
\end{equation}
\par The relation between curvature tensor $\tilde{R}(X, Y)Z$ of $M^n$ with respect to quarter-symmetric
metric connection $\tilde{\nabla}$ and the curvature tensor $R(X, Y)Z$ with respect to the Levi-Civita connection $\nabla$ is given by
\begin{eqnarray}
\nonumber \tilde{R}(X, Y)Z= R(X, Y)Z + 3g(\phi X, Z)\phi Y - 3g(\phi Y, Z)\phi X \\
+ \{\eta(X)Y-\eta(Y)X \}\eta(Z) -  [g(Y, Z) \eta(X) -\eta(Y) g(X, Z)]\xi  \label{1.12}.
\end{eqnarray}
Also from (\ref{1.12}) we
obtain
\begin{equation}
 \tilde{S}(Y, Z)= S(Y, Z) + 2g(Y, Z) - (n+1)\eta(Y)\eta(Z) - 3trace\phi\, g(\phi Y, Z) \label{1.13},
\end{equation}
where $\tilde{S}$ and $S$ are Ricci tensors of connections $\tilde{\nabla}$ and $\nabla$ respectively.

\section{\bf{Pseudosymmetric para-Sasakian manifold admitting a quarter-symmetric metric connection\label{sec3}}}

A para-Sasakian manifold $M^{n}$ admitting a quarter-symmetric metric connection is said to be pseudosymmetric if
\begin{equation}
 \tilde{R}(X, Y)\cdot \tilde{R}=L_{\tilde{R}}[(X\wedge_{g} Y)\cdot \tilde{R}], \label{3.1}
\end{equation}
holds on the set $U_{\tilde{R}} = \{x \in M^{n} : \tilde{R} \neq 0 \,\,at\,\, x$\}, where $L_{\tilde{R}}$ is some function on $U_{\tilde{R}}$.

Suppose that $M^{n}$ be pseudosymmetric, then in view of (\ref{3.1}) we have
\begin{eqnarray}
\nonumber \tilde{R}(\xi, Y)\tilde{R}(U, V)W - \tilde{R}(\tilde{R}(\xi, Y)U, V)W - \tilde{R}(U, \tilde{R}(\xi, Y)V)W \\
\nonumber - \tilde{R}(U, V)\tilde{R}(\xi, Y)W =L_{\tilde{R}}[(\xi\wedge_{g} Y)\tilde{R}(U, V)W - \tilde{R}((\xi\wedge_{g} Y)U, V)W \\
- \tilde{R}(U, (\xi\wedge_{g} Y)V)W - \tilde{R}(U, V)(\xi\wedge_{g} Y)W]. \label{3.2}
\end{eqnarray}

By virtue of (\ref{1.8}) and (\ref{1.12}), (\ref{3.2}) takes the form
\begin{eqnarray}
\nonumber (L_{\tilde{R}}+2)[ \eta(\tilde{R}(U, V)W)Y- g(Y, \tilde{R}(U, V)W)\xi - \eta(U)\tilde{R}(Y, V)W + g(Y, U)\tilde{R}(\xi, V)W\\
- \eta(V)\tilde{R}(U, Y)W + g(Y, V)\tilde{R}(U, \xi)W- \eta(W)\tilde{R}(U, V)Y + g(Y, W)\tilde{R}(U, V)\xi] = 0. \label{3.3}
\end{eqnarray}

Taking inner product of (\ref{3.3}) with $\xi$ and using (\ref{1.6}) and (\ref{1.12}), we get
\begin{eqnarray}
\nonumber (L_{\tilde{R}}+2)[ g(Y, R(U, V)W) + 3g(\phi U, W)g(\phi V, Y) - 3g(\phi V, W)g(\phi U, Y) \\
\nonumber + \eta(W)\{\eta(U)g(V, Y)-\eta(V)g(U, Y) \} -  \{g(V, W) \eta(U) -\eta(V) g(U, W)\}\eta(Y) \\
+ 2 \{g(V, W) g(Y, U) - g(V, Y) g(U, W)\}]=0.   \label{3.4}
\end{eqnarray}

Assuming that $L_{\tilde{R}}+2\neq 0$, the above equation becomes
\begin{eqnarray}
\nonumber  g(Y, R(U, V)W) + 3g(\phi U, W)g(\phi V, Y) - 3g(\phi V, W)g(\phi U, Y) \\
\nonumber + \eta(W)\{\eta(U)g(V, Y)-\eta(V)g(U, Y) \} -  [g(V, W) \eta(U) -\eta(V) g(U, W)]\eta(Y) \\
+ 2 [g(V, W) g(Y, U) - g((V, Y) g(U, W)] = 0. \label{3.5}
\end{eqnarray}

Putting $V=W=e_{i}$, where $\{e_{i}\}$ is an orthonormal basis of the tangent space at each point of the
manifold and taking summation over $i$, $i = 1, 2, 3,\cdots, n$, we get
\begin{equation}
\tilde{S}(Y, U)= -2(n-1) g(Y, U). \label{3.6}
\end{equation}

Hence, we can state the following:
\begin{thm}\label{thm1}
 If a para-Sasakian manifold $M^n$ admitting a quarter-symmetric metric connection is pseudosymmetric then $M^n$ is an Einstein manifold with respect to quarter-symmetric metric connection or it satisfies $L_{\tilde{R}}=-2$.
\end{thm}

\section{\bf{Ricci-pseudosymmetric para-Sasakian manifold admitting a quarter-symmetric metric connection\label{sec4}}}
A para-Sasakian manifold $M^n$ admitting a quarter-symmetric metric connection is said to be Ricci-pseudosymmetric if the following condition is satisfied
\begin{equation}
\tilde{R} \cdot \tilde{S}= L_{\tilde{S}}[ (X\wedge_{g} Y)\cdot \tilde{S}],    \label{4.1}
\end{equation}
on $U_{\tilde{S}}$.
\par Let para-Sasakian manifold $M^{n}$ admitting a quarter-symmetric metric connection be Ricci-pseudosymmetric. Then we have
\begin{equation}
\tilde{S}(\tilde{R}(X, Y)Z, W) + \tilde{S}(Z, \tilde{R}(X, Y)W) = L_{\tilde{S}}[ \tilde{S}((X\wedge_{g} Y)Z, W) + \tilde{S}(Z, (X\wedge_{g} Y)W) ].    \label{4.2}
\end{equation}

By taking $Y=W=\xi$ and making use of (\ref{1.8}), (\ref{1.9}) and (\ref{1.12}), the above equation turns into
\begin{equation}
 (L_{\tilde{S}}+2)[ \tilde{S}(X, Z) + 2(n-1) g(X, Z) ]=0   \label{4.3}
\end{equation}

Thus, we have the following assertion:
\begin{thm}\label{thm2}
 If a para-Sasakian manifold $M^n$ admitting a quarter-symmetric metric connection is Ricci pseudosymmetric then $M^n$ is an Einstein manifold with respect to quarter-symmetric metric connection or it satisfies $L_{\tilde{S}}=-2$.
\end{thm}

\section{\bf{Projectively flat para-Sasakian manifold admitting a quarter-symmetric metric connection\label{sec5}}}
The projective curvature tensor on a Riemannian manifold is defined by
\begin{equation}
 P(X, Y)Z=R(X, Y)Z - \frac{1}{(n-1)}[S(Y, Z)X - S(X, Z)Y].   \label{5.1}
\end{equation}

For an $n$-dimensional para-Sasakian manifold $M^{n}$ admitting a quarter-symmetric metric connection, the projective curvature tensor is given by
\begin{equation}
 \tilde{P}(X, Y)Z= \tilde{R}(X, Y)Z - \frac{1}{(n-1)}[ \tilde{S}(Y, Z)X -  \tilde{S}(X, Z)Y].   \label{5.2}
\end{equation}

\begin{thm}\label{thm3}
A projectively flat para-Sasakian manifold $M^n$ admitting a quarter-symmetric metric connection is an Einstein manifold with respect to quarter-symmetric metric connection.
\end{thm}
\begin{proof}
Consider a projectively flat para-Sasakian manifold admitting a quarter-symmetric metric connection. Then from (\ref{5.2}) we have
\begin{equation}
 g(\tilde{R}(X, Y)Z, W) = \frac{1}{(n-1)}[ \tilde{S}(Y, Z)g(X, W) -  \tilde{S}(X, Z)g(Y, W)].   \label{5.3}
\end{equation}

Setting $X=W=\xi$ in (\ref{5.3}) and using (\ref{1.8}), (\ref{1.9}), (\ref{1.12}) and (\ref{1.13}), we get
\begin{equation}
 \tilde{S}(X, Z) =- 2(n-1) g(X, Z).    \label{5.4}
\end{equation}
Hence the proof.
\end{proof}

\section{\bf{Projectively pseudosymmetric para-Sasakian manifold admitting a quarter-symmetric metric connection\label{sec6}}}
A para-Sasakian manifold admitting a quarter-symmetric metric connection is said to be projectively pseudosymmetric if
\begin{equation}
 \tilde{R}(X, Y)\cdot \tilde{P}=L_{\tilde{P}}[(X\wedge_{g} Y)\cdot \tilde{P}], \label{6.1}
\end{equation}
holds on the set $U_{\tilde{P}} = \{x \in M^{n} : \tilde{P} \neq 0\}$ at $x$, where $L_{\tilde{P}}$ is some function on $U_{\tilde{P}}$.

Let $M^n$ be projectively pseudosymmetric, then we have
\begin{eqnarray}
\nonumber \tilde{R}(X, \xi)\tilde{P}(U, V)\xi - \tilde{P}(\tilde{R}(X, \xi)U, V)\xi - \tilde{P}(U, \tilde{R}(X, \xi)V)\xi \\
\nonumber - \tilde{P}(U, V)\tilde{R}(X, \xi)\xi =L_{\tilde{P}}[(X\wedge_{g} \xi)\tilde{P}(U, V)\xi - \tilde{P}((X\wedge_{g} \xi)U, V)\xi \\
- \tilde{P}(U, (X\wedge_{g} \xi)V)\xi - \tilde{P}(U, V)(X\wedge_{g} \xi)\xi]. \label{6.2}
\end{eqnarray}

By virtue of (\ref{1.12}), (\ref{1.13}) and (\ref{5.2}), (\ref{6.2}) becomes
\begin{equation}
 (L_{\tilde{P}}+2)\tilde{P}(U, V)X=0. \label{6.3}
\end{equation}

So, one can state that:
\begin{thm}\label{thm4}
If a para-Sasakian manifold $M^n$ admitting a quarter-symmetric metric connection is projectively pseudosymmetric then $M^n$ is projectively flat with respect to quarter-symmetric metric connection or $L_{\tilde{P}}=-2$.
\end{thm}

In view of theorem \ref{thm3}, one can state the above theorem as
\begin{thm}\label{thm4}
 If a para-Sasakian manifold $M^n$ admitting a quarter-symmetric metric connection is projectively pseudosymmetric then $M^n$ is an Einstein manifold with respect to quarter-symmetric metric connection or $L_{\tilde{P}}=-2$.
\end{thm}

\section{\bf Example \label{sec7}}
\subsection{\bf Example}
\par We consider a 3-dimensional manifold $M = \{(x, y, z)\in \mathbb{R}^3: z\neq 0\}$, where $(x, y, z)$
are standard coordinates in $\mathbb{R}^3$. Let $\{E_{1}, E_{2}, E_{3}\}$ be a linearly independent global frame field on $M$ given by
\begin{equation}
\nonumber E_{1}= e^{z}\frac{\partial}{\partial{y}}, \,\,\,\,\,\,\,\,\,\,\, E_{2}= e^{z}(\frac{\partial}{\partial{y}}-\frac{\partial}{\partial{x}}), \,\,\,\,\,\,\,\,\,\, E_{3}= \frac{\partial}{\partial{z}},
\end{equation}

\par If $g$ is a Riemannian metric defined by
\[
    g(E_{i}, E_{j})=
\begin{cases}
    1, & i=j\\
    0, & i\neq j
\end{cases}
\]
for $1\leq i, j\leq 3,$ and if $\eta$ is the 1-form defined by $\eta(Z)=g(Z, E_{3})$ for any vector field $Z\in \chi(M)$. We define the (1, 1)-tensor field $\phi$ as
\begin{equation}
\nonumber \phi (E_{1})= E_{1}, \,\,\,\,\,\, \phi (E_{2})= -E_{2}, \,\,\,\,\, \phi (E_{3})= 0.
\end{equation}

The linearity property of $\phi$ and $g$ yields that
\begin{eqnarray}
\nonumber & & \eta(E_{3})=1, \\
\nonumber & & \phi^{2}U=U - \eta(U)E_{3}, \\
\nonumber & & g(\phi U, \phi V)=g(U, V)-\eta(U) \eta(V),
\end{eqnarray}
for any $U, V \in \chi(M)$.

\par Now we have
\begin{equation}
\nonumber [E_{1}, E_{2}]=0, \,\,\,\,\,\, [E_{1}, E_{3}]= E_{1}, \,\,\,\,\, [E_{2}, E_{3}]= E_{2}.
\end{equation}

The Riemannian connection $\nabla$ of the metric $g$ known as Koszul's formula and is given by
\begin{eqnarray}
\nonumber 2g(\nabla_{X}Y, Z) = Xg(Y, Z)+Yg(Z, X)-Zg(X, Y)-g(X, [Y, Z]) \\
\nonumber - g(Y, [X, Z]) + g(Z, [X, Y]).
\end{eqnarray}
Using Koszul's formula we get the followings in matrix form
\begin{equation}
\nonumber \left(
  \begin{array}{ccc}
   \nabla_{E_{1}}E_{1} &\nabla_{E_{1}}E_{2} &\nabla_{E_{1}}E_{3} \\
   \nabla_{E_{2}}E_{1} &\nabla_{E_{2}}E_{2} &\nabla_{E_{2}}E_{3}  \\
   \nabla_{E_{3}}E_{1} &\nabla_{E_{3}}E_{2} &\nabla_{E_{3}}E_{3} \\
  \end{array}
\right) =\left(
           \begin{array}{ccc}
             - E_{3} & 0 & E_{1} \\
             0 & -E_{3} & E_{2}  \\
              0 & 0 & 0 \\
           \end{array}
         \right).
\end{equation}

\par Clearly $(\phi, \xi, \eta, g)$ is a para-Sasakian structure on $M$. Thus $M(\phi, \xi, \eta, g)$ is a 3-dimensional para-Sasakian manifold.

Using (\ref{1.11}) and the above equation, one can easily obtain the following:
\begin{equation}
\nonumber \left(
  \begin{array}{ccc}
    \tilde{\nabla}_{E_{1}}E_{1} & \tilde{\nabla}_{E_{1}}E_{2} & \tilde{\nabla}_{E_{1}}E_{3} \\
    \tilde{\nabla}_{E_{2}}E_{1} & \tilde{\nabla}_{E_{2}}E_{2} & \tilde{\nabla}_{E_{2}}E_{3}  \\
    \tilde{\nabla}_{E_{3}}E_{1} & \tilde{\nabla}_{E_{3}}E_{2} & \tilde{\nabla}_{E_{3}}E_{3} \\
  \end{array}
\right) =\left(
           \begin{array}{ccc}
             - 2E_{3} & 0 & 2E_{1} \\
             0 & -2E_{3} & 2E_{2}  \\
              0 & 0 & 0 \\
           \end{array}
         \right).
\end{equation}

\par With the help of the above results it can be easily verified that
\begin{eqnarray}
\nonumber & & R(E_{1}, E_{2})E_{3} = 0, \,\,\,\,\,\,\,\,\,\,\,\,\,\,\,\,\,\,\,\,\,\, R(E_{2}, E_{3})E_{3}=-E_{2} , \,\,\,\,\,\,\,\,\,\,\,\,\,\,\,\, R(E_{1}, E_{3})E_{3}=-E_{1} ,\\
\nonumber & & R(E_{1}, E_{2})E_{2}=-E_{1}, \,\,\,\,\,\,\,\,\,\,\,\,\, R(E_{2}, E_{3})E_{2}=E_{3} , \,\,\,\,\,\,\,\,\,\,\,\,\,\,\,\,\,\,\,\,\,\, R(E_{1}, E_{3})E_{2} =0 ,\\
\nonumber & & R(E_{1}, E_{2})E_{1} =E_{2}, \,\,\,\,\,\,\,\,\,\,\,\,\,\,\,\,\, R(E_{2}, E_{3})E_{1} =0 , \,\,\,\,\,\,\,\,\,\,\,\,\,\,\,\,\,\,\,\,\,\,\,\,\,\,\, R(E_{1}, E_{3})E_{1} =E_{3}.
\end{eqnarray}
 and
 \begin{eqnarray}
\nonumber & & \tilde{R}(E_{1}, E_{2})E_{3} = 0, \,\,\,\,\,\,\,\,\,\,\,\,\,\,\,\,\,\,\,\,\,\,\,\, \tilde{R}(E_{2}, E_{3})E_{3}=-2E_{2} , \,\,\,\,\,\,\,\,\,\,\,\,\,\,\,\,\,\, \tilde{R}(E_{1}, E_{3})E_{3}=-2E_{1}, \\
\nonumber & & \tilde{R}(E_{1}, E_{2})E_{2}=-4E_{1}, \,\,\,\,\,\,\,\,\,\,\,\, \tilde{R}(E_{2}, E_{3})E_{2}=2E_{3} , \,\,\,\,\,\,\,\,\,\,\,\,\,\,\,\,\,\,\,\,\, \tilde{R}(E_{1}, E_{3})E_{2} =0,\\
 & & \tilde{R}(E_{1}, E_{2})E_{1} =4E_{2}, \,\,\,\,\,\,\,\,\,\,\,\,\,\,\,\,\, \tilde{R}(E_{2}, E_{3})E_{1} =0 , \,\,\,\,\,\,\,\,\,\,\,\,\,\,\,\,\,\,\,\,\,\,\,\,\,\,\, \tilde{R}(E_{1}, E_{3})E_{1} =2E_{3}.\label{7.1}
\end{eqnarray}

\par Since ${E_{1}, E_{2}, E_{3}}$ forms a basis, any vector field $X, Y, Z \in \chi(M)$ can be
written as $X = a_{1}E_{1} + b_{1}E_{2} + c_{1}E_{3}$, $Y = a_{2}E_{1} + b_{2}E_{2} + c_{2}E_{3}$, $Z= a_{3}E_{1} + b_{3}E_{2} + c_{3}E_{3}$, where $a_{i}, b_{i}, c_{i} \in \mathbb{R}^{+}$ (the set of all positive real numbers), $i = 1, 2, 3$. Using the expressions of the curvature tensors, we find values of Riemannian curvature and Ricci curvature with respect to quarter-symmetric metric connection as;
\begin{eqnarray}
\nonumber \tilde{R}(X, Y)Z &=& [-4\{a_{1}b_{2}-b_{1}a_{2}\} b_{3}  + 2 \{c_{1}a_{2}-a_{1}c_{2}\} c_{3}] E_{1}\\
\nonumber &+& [-4\{b_{1}a_{2}-a_{1}b_{2}\} a_{3} + 2\{c_{1}b_{2}-b_{1}c_{2}\} c_{3}] E_{2}\\
 &+& [ -2\{c_{1}a_{2}-a_{1}c_{2}\} a_{3} - 2\{c_{1}b_{2}-b_{1}c_{2}\} b_{3}] E_{3},  \label{7.2}\\
\tilde{S}(E_{1}, E_{1}) &=& \tilde{S}(E_{2}, E_{2})=-6, \,\, \tilde{S}(E_{3}, E_{3})=-4.\label{7.3}
\end{eqnarray}

Using (\ref{7.1}), (\ref{7.3}) and the expression of the endomorphism $(X\wedge_{g} Y)Z=g(Y, Z)X-g(X, Z)Y$, one can easily verify that
\begin{equation}
\tilde{S}(\tilde{R}(X, E_{3})Y, E_{3}) + \tilde{S}(Y, \tilde{R}(X, E_{3})E_{3}) = -2[ \tilde{S}((X\wedge_{g} E_{3})Y, E_{3}) + \tilde{S}(Y, (X\wedge_{g} E_{3})E_{3}) ],     \label{7.4}
\end{equation}
here $L_{\tilde{S}}=-2$. Thus, the above equation verify one part of the theorem \ref{thm2} of section \ref{sec4}

Moreover, the manifold under consideration satisfies
\begin{eqnarray}
\nonumber \tilde{R}(X, Y)Z=-\tilde{R}(Y, X)Z,\\
\nonumber \tilde{R}(X, Y)Z+\tilde{R}(Y, Z)X+\tilde{R}(Z, X)Y=0.
\end{eqnarray}
Hence, from the above equations one can say that this example verifies the condition $c$ of theorem 3.1 in \cite{AKMUCD} and first Bianchi identity.

\subsection{\bf Example}

\par We consider a 5-dimensional manifold $M = \{(x_{1}, x_{2}, x_{3}, x_{4}, x_{5})\in \mathbb{R}^5\}$, where $(x_{1}, x_{2}, x_{3}, x_{4}, x_{5})$
are standard coordinates in $\mathbb{R}^5$. We choose the vector fields
\begin{equation}
\nonumber E_{1}= \frac{\partial}{\partial{x_{1}}}, \,\,\,\, E_{2}= \frac{\partial}{\partial{x_{2}}}, \,\,\,\, E_{3}= \frac{\partial}{\partial{x_{3}}}, \,\,\,\, E_{4}= \frac{\partial}{\partial{x_{4}}}, \,\,\,\, E_{5}=x_{1}\frac{\partial}{\partial{x_{1}}}+ x_{2}\frac{\partial}{\partial{x_{2}}} + x_{3}\frac{\partial}{\partial{x_{3}}} + x_{4}\frac{\partial}{\partial{x_{4}}} + \frac{\partial}{\partial{x_{5}}},
\end{equation}
which are linearly independent at each point of $M$

\par Let $g$ be a Riemannian metric defined by
\[
    g(E_{i}, E_{j})=
\begin{cases}
    1, & i=j\\
    0, & i\neq j
\end{cases}
\]
for $1\leq i, j\leq 5,$ and if $\eta$ is the 1-form defined by $\eta(Z)=g(Z, E_{5})$ for any vector field $Z\in \chi(M)$. Let $\phi$ be the (1, 1)-tensor field defined by
\begin{equation}
\nonumber \phi (E_{1})= E_{1}, \,\,\,\,\, \phi (E_{2})= E_{2}, \,\,\,\,\, \phi (E_{3})= E_{3}, \,\,\,\,\, \phi (E_{4})= E_{4}, \,\,\,\,\, \phi (E_{5})= 0.
\end{equation}

The linearity property of $\phi$ and $g$ yields that
\begin{eqnarray}
\nonumber & & \eta(E_{5})=1, \\
\nonumber & & \phi^{2}U=U - \eta(U)E_{5}, \\
\nonumber & & g(\phi U, \phi V)=g(U, V)-\eta(U) \eta(V),
\end{eqnarray}
for any $U, V \in \chi(M)$.

\par Now we have
\begin{eqnarray*}
 && [E_{1}, E_{2}]=0, \,\,\,\,\,\,\,\, [E_{1}, E_{3}]=0, \,\,\,\,\,\,\,\, [E_{1}, E_{4}]=0, \,\,\,\,\,\,\,\, [E_{1}, E_{5}]=E_{1},\\
 && [E_{2}, E_{3}]= 0, \,\,\,\,\,\,\,\, [E_{2}, E_{4}]= 0, \,\,\,\,\,\,\,\, [E_{2}, E_{5}]= E_{2},\\
 && [E_{3}, E_{4}]=0, \,\,\,\,\,\,\,\, [E_{3}, E_{5}]=E{3}, \,\,\,\,\,\,\,\, [E_{4}, E_{5}]=E_{4}.
\end{eqnarray*}

By virtue of Koszul's formula we get the followings in matrix form
\begin{equation}
\nonumber \left(
  \begin{array}{ccccc}
   \nabla_{E_{1}}E_{1} &\nabla_{E_{1}}E_{2} &\nabla_{E_{1}}E_{3} &\nabla_{E_{1}}E_{4} &\nabla_{E_{1}}E_{5}\\
   \nabla_{E_{2}}E_{1} &\nabla_{E_{2}}E_{2} &\nabla_{E_{2}}E_{3} &\nabla_{E_{2}}E_{4} &\nabla_{E_{2}}E_{5} \\
   \nabla_{E_{3}}E_{1} &\nabla_{E_{3}}E_{2} &\nabla_{E_{3}}E_{3} &\nabla_{E_{3}}E_{4} &\nabla_{E_{3}}E_{5} \\
  \nabla_{E_{4}}E_{1} &\nabla_{E_{4}}E_{2} &\nabla_{E_{4}}E_{3} &\nabla_{E_{4}}E_{4} &\nabla_{E_{4}}E_{5} \\
   \nabla_{E_{5}}E_{1} &\nabla_{E_{5}}E_{2} &\nabla_{E_{5}}E_{3} &\nabla_{E_{5}}E_{4} &\nabla_{E_{5}}E_{5} \\
  \end{array}
\right) =\left(
           \begin{array}{ccccc}
             - E_{5} & 0 & 0 & 0 & E_{1} \\
             0 & -E_{5} & 0 & 0 & E_{2}  \\
             0 & 0 & -E_{5} & 0 & E_{3} \\
              0 & 0 & 0 & -E_{5} & E_{4} \\
              0 & 0 & 0 & 0 & 0   \\
           \end{array}
         \right).
\end{equation}

\par Above expressions satisfies all the properties of para-Sasakian manifold. Thus $M(\phi, \xi, \eta, g)$ is a 5-dimensional para-Sasakian manifold.

From the above expressions and the relation of quarter symmetric metric connection and Riemannian connection, one can easily obtain the following:
\begin{equation}
\nonumber \left(
  \begin{array}{ccccc}
   \tilde{\nabla}_{E_{1}}E_{1} &\tilde{\nabla}_{E_{1}}E_{2} &\tilde{\nabla}_{E_{1}}E_{3} &\tilde{\nabla}_{E_{1}}E_{4} &\tilde{\nabla}_{E_{1}}E_{5}\\
   \tilde{\nabla}_{E_{2}}E_{1} &\tilde{\nabla}_{E_{2}}E_{2} &\tilde{\nabla}_{E_{2}}E_{3} &\tilde{\nabla}_{E_{2}}E_{4} &\tilde{\nabla}_{E_{2}}E_{5} \\
   \tilde{\nabla}_{E_{3}}E_{1} &\tilde{\nabla}_{E_{3}}E_{2} &\tilde{\nabla}_{E_{3}}E_{3} &\tilde{\nabla}_{E_{3}}E_{4} &\tilde{\nabla}_{E_{3}}E_{5} \\
  \tilde{\nabla}_{E_{4}}E_{1} &\tilde{\nabla}_{E_{4}}E_{2} &\tilde{\nabla}_{E_{4}}E_{3} &\tilde{\nabla}_{E_{4}}E_{4} &\tilde{\nabla}_{E_{4}}E_{5} \\
   \tilde{\nabla}_{E_{5}}E_{1} &\tilde{\nabla}_{E_{5}}E_{2} &\tilde{\nabla}_{E_{5}}E_{3} &\tilde{\nabla}_{E_{5}}E_{4} &\tilde{\nabla}_{E_{5}}E_{5} \\
  \end{array}
\right) =\left(
           \begin{array}{ccccc}
             - 2E_{5} & 0 & 0 & 0 & 2E_{1} \\
             0 & -2E_{5} & 0 & 0 & 2E_{2}  \\
             0 & 0 & -2E_{5} & 0 & 2E_{3} \\
              0 & 0 & 0 & -2E_{5} & 2E_{4} \\
              0 & 0 & 0 & 0 & 0   \\
           \end{array}
         \right).
\end{equation}

\par With the help of the above results it can be easily obtain the non-zero components of curvature tensors as
\begin{eqnarray}
\nonumber & & R(E_{1}, E_{2})E_{1} = E_{2}, \,\,\,\,\, R(E_{1}, E_{2})E_{2} = -E_{1}, \,\,\,\,\, R(E_{1}, E_{3})E_{1} = E_{3}, \,\,\,\,\, R(E_{1}, E_{3})E_{3} = -E_{1},\\
\nonumber & & R(E_{1}, E_{4})E_{1} = E_{4}, \,\,\,\,\, R(E_{1}, E_{4})E_{4} = -E_{1}, \,\,\,\,\, R(E_{1}, E_{5})E_{1} = E_{5}, \,\,\,\,\, R(E_{1}, E_{5})E_{5} = -E_{1},\\
\nonumber & & R(E_{2}, E_{3})E_{2} = E_{3}, \,\,\,\,\, R(E_{2}, E_{3})E_{3} = -E_{2}, \,\,\,\,\, R(E_{2}, E_{4})E_{2} = E_{4}, \,\,\,\,\, R(E_{2}, E_{4})E_{4} = -E_{2},\\
\nonumber & & R(E_{2}, E_{5})E_{2} = E_{5}, \,\,\,\,\, R(E_{2}, E_{5})E_{5} = -E_{2}, \,\,\,\,\, R(E_{3}, E_{4})E_{3} = E_{4}, \,\,\,\,\, R(E_{3}, E_{4})E_{4} = -E_{3},\\
\nonumber & & R(E_{3}, E_{5})E_{3} = E_{5}, \,\,\,\,\, R(E_{3}, E_{5})E_{5} = -E_{3}, \,\,\,\,\, R(E_{4}, E_{5})E_{4} = E_{5}, \,\,\,\,\, R(E_{4}, E_{5})E_{5} = -E_{4},
\end{eqnarray}
 and
\begin{eqnarray}
\nonumber & & \tilde{R}(E_{1}, E_{2})E_{1} = 4E_{2}, \,\,\,\,\, \tilde{R}(E_{1}, E_{2})E_{2} = -4E_{1}, \,\,\,\,\, \tilde{R}(E_{1}, E_{3})E_{1} = 4E_{3}, \,\,\,\,\, \tilde{R}(E_{1}, E_{3})E_{3} = -4E_{1},\\
\nonumber & & \tilde{R}(E_{1}, E_{4})E_{1} = 4E_{4}, \,\,\,\,\, \tilde{R}(E_{1}, E_{4})E_{4} = -4E_{1}, \,\,\,\,\, \tilde{R}(E_{1}, E_{5})E_{1} = 2E_{5}, \,\,\,\,\,\,\,\,
 \tilde{R}(E_{1}, E_{5})E_{5} = -2E_{1},\\
\nonumber & & \tilde{R}(E_{2}, E_{3})E_{2} = 4E_{3}, \,\,\,\,\, \tilde{R}(E_{2}, E_{3})E_{3} = -4E_{2}, \,\,\,\,\, \tilde{R}(E_{2}, E_{4})E_{2} = 4E_{4}, \,\,\,\,\, \tilde{R}(E_{2}, E_{4})E_{4} = -4E_{2},\\
\nonumber & & \tilde{R}(E_{2}, E_{5})E_{2} = 2E_{5}, \,\,\,\,\, \tilde{R}(E_{2}, E_{5})E_{5} = -2E_{2}, \,\,\,\,\, \tilde{R}(E_{3}, E_{4})E_{3} = 4E_{4}, \,\,\,\,\, \tilde{R}(E_{3}, E_{4})E_{4} = -4E_{3},\\
 & & \tilde{R}(E_{3}, E_{5})E_{3} = 2E_{5}, \, \tilde{R}(E_{3}, E_{5})E_{5} = -2E_{3}, \, \tilde{R}(E_{4}, E_{5})E_{4} = 2E_{5}, \, \tilde{R}(E_{4}, E_{5})E_{5} = -2E_{4} \label{8.1}.
\end{eqnarray}

\par Since ${E_{1}, E_{2}, E_{3}, E_{4}, E_{5}}$ forms a basis, any vector field $X, Y, Z \in \chi(M)$ can be
written as $X = a_{1}E_{1} + b_{1}E_{2} + c_{1}E_{3} +  d_{1}E_{4} + f_{1}E_{5}$, $Y = a_{2}E_{1} + b_{2}E_{2} + c_{2}E_{3} + d_{2}E_{4} + f_{2}E_{5}$, $Z= a_{3}E_{1} + b_{3}E_{2} + c_{3}E_{3} + d_{3}E_{4} + f_{3}E_{5}$, where $a_{i}, b_{i}, c_{i}, d_{i}, f_{i} \in \mathbb{R}^{+}$ (the set of all positive real numbers), $i = 1, 2, 3, 4, 5$. Using the expressions of the curvature tensors, we find values of Riemannian curvature and Ricci curvature with respect to quarter-symmetric metric connection as;
\begin{eqnarray}
\nonumber \tilde{R}(X, Y)Z &=& [-4\{a_{1}(b_{2}b_{3} + c_{2}c_{3} + d_{2}d_{3}) - a_{2}(b_{1}b_{3} + c_{1}c_{3} + d_{1}d_{3})\} -2 (a_{1}f_{2} - f_{1}a_{2})f_{3}] E_{1}\\
\nonumber &+& [-4\{b_{1}(a_{2}a_{3} + c_{2}c_{3} + d_{2}d_{3}) - b_{2}(a_{1}a_{3} + c_{1}c_{3} + d_{1}d_{3})\} -2 (b_{1}f_{2} - f_{1}b_{2})f_{3}] E_{2}\\
\nonumber &+& [-4\{c_{1}(a_{2}a_{3} + b_{2}b_{3} + d_{2}d_{3}) - c_{2}(a_{1}a_{3} + b_{1}b_{3} + d_{1}d_{3})\} -2 (c_{1}f_{2} - f_{1}c_{2})f_{3}] E_{3}\\
\nonumber &+& [-4\{d_{1}(a_{2}a_{3} + b_{2}b_{3} + c_{2}c_{3}) - d_{2}(a_{1}a_{3} + b_{1}b_{3} + c_{1}c_{3})\} -2 (d_{1}f_{2} - f_{1}d_{2})f_{3}] E_{4}\\
\nonumber &+& [2\{(a_{1}f_{2} - f_{1}a_{2})a_{3} + (b_{1}f_{2} - f_{1}b_{2})b_{3} + (c_{1}f_{2} - f_{1}c_{2})c_{3} + (d_{1}f_{2} - f_{1}d_{2})d_{3}\}] E_{5},  \label{8.2}\\
\tilde{S}(E_{1}, E_{1}) &=& \tilde{S}(E_{2}, E_{2})=\tilde{S}(E_{3}, E_{3})=\tilde{S}(E_{4}, E_{4})=-14, \,\, \tilde{S}(E_{5}, E_{5})=-8.\label{8.3}
\end{eqnarray}

In view of (\ref{8.1}), (\ref{8.3}) and the expression of the endomorphism one can easily verify the equation (\ref{7.4}) and hence the theorem \ref{thm2} of section \ref{sec4} is verified. This example also verifies the condition $c$ of theorem 3.1 in \cite{AKMUCD} and first Bianchi identity.

\par Above two examples verifies the one part of the theorem \ref{thm2}, that is, if a para-Sasakian manifold $M^n$ admitting a quarter-symmetric metric connection is Ricci pseudosymmetric then $M^n$ satisfies $L_{\tilde{S}}=-2$ ($M^n$ is not Einstein manifold with respect to quarter-symmetric metric connection). Another part of the theorem is that, if a para-Sasakian manifold $M^n$ admitting a quarter-symmetric metric connection is Ricci pseudosymmetric then $M^n$ is an Einstein manifold with respect to quarter-symmetric metric connection ($L_{\tilde{S}}\neq-2$). Now, the second part of the theorem \ref{thm2} can be verified by using the proper example.

\subsection{\bf Example}
\par We consider a 5-dimensional manifold $M = \{(x, y, z, u, v)\in \mathbb{R}^5\}$, where $(x, y, z, u, v)$
are standard coordinates in $\mathbb{R}^5$. Let $\{E_{1}, E_{2}, E_{3}, E_{4}, E_{5}\}$ be a linearly independent global frame field on $M$ given by
\begin{equation}
\nonumber E_{1}= \frac{\partial}{\partial{x}}, \,\,\,\, E_{2}=e^{-x} \frac{\partial}{\partial{y}}, \,\,\,\, E_{3}=e^{-x} \frac{\partial}{\partial{z}}, \,\,\,\, E_{4}=e^{-x} \frac{\partial}{\partial{u}}, \,\,\,\, E_{5}=e^{-x}\frac{\partial}{\partial{v}}.
\end{equation}

\par Let $g$ be a Riemannian metric defined by
\[
    g(E_{i}, E_{j})=
\begin{cases}
    1, & i=j\\
    0, & i\neq j
\end{cases}
\]
for $1\leq i, j\leq 5,$ and if $\eta$ is the 1-form defined by $\eta(Z)=g(Z, E_{1})$ for any vector field $Z\in \chi(M)$. Let the (1, 1)-tensor field $\phi$ be defined by
\begin{equation}
\nonumber \phi (E_{1})= 0, \,\,\,\,\, \phi (E_{2})= E_{2}, \,\,\,\,\, \phi (E_{3})= E_{3}, \,\,\,\,\, \phi (E_{4})= E_{4}, \,\,\,\,\, \phi (E_{5})= E_{5}.
\end{equation}

With the help of linearity property of $\phi$ and $g$, we have
\begin{eqnarray}
\nonumber & & \eta(E_{1})=1, \\
\nonumber & & \phi^{2}V=V - \eta(V)E_{1}, \\
\nonumber & & g(\phi X, \phi Y)=g(X, Y)-\eta(X) \eta(Y),
\end{eqnarray}
for any $X, Y \in \chi(M)$.

\par Now we have
\begin{eqnarray*}
 && [E_{1}, E_{2}]=-E_{2}, \,\,\,\,\,\,\,\, [E_{1}, E_{3}]=-E_{3}, \,\,\,\,\,\,\,\, [E_{1}, E_{4}]=-E_{4}, \,\,\,\,\,\,\,\, [E_{1}, E_{5}]=-E_{5},\\
 && [E_{2}, E_{3}] = [E_{2}, E_{4}]= [E_{2}, E_{5}]= [E_{3}, E_{4}]= [E_{3}, E_{5}]= E_{4}, E_{5}]=0.
\end{eqnarray*}

With the help of Koszul's formula we get the followings in matrix form
\begin{equation}
\nonumber \left(
  \begin{array}{ccccc}
   \nabla_{E_{1}}E_{1} &\nabla_{E_{1}}E_{2} &\nabla_{E_{1}}E_{3} &\nabla_{E_{1}}E_{4} &\nabla_{E_{1}}E_{5}\\
   \nabla_{E_{2}}E_{1} &\nabla_{E_{2}}E_{2} &\nabla_{E_{2}}E_{3} &\nabla_{E_{2}}E_{4} &\nabla_{E_{2}}E_{5} \\
   \nabla_{E_{3}}E_{1} &\nabla_{E_{3}}E_{2} &\nabla_{E_{3}}E_{3} &\nabla_{E_{3}}E_{4} &\nabla_{E_{3}}E_{5} \\
  \nabla_{E_{4}}E_{1} &\nabla_{E_{4}}E_{2} &\nabla_{E_{4}}E_{3} &\nabla_{E_{4}}E_{4} &\nabla_{E_{4}}E_{5} \\
   \nabla_{E_{5}}E_{1} &\nabla_{E_{5}}E_{2} &\nabla_{E_{5}}E_{3} &\nabla_{E_{5}}E_{4} &\nabla_{E_{5}}E_{5} \\
  \end{array}
\right) =\left(
           \begin{array}{ccccc}
             0 & 0 & 0 & 0 & 0 \\
             E_{2} & -E_{1} & 0 & 0 & 0  \\
             E_{3} & 0 & -E_{1} & 0 & 0 \\
              E_{4} & 0 & 0 & -E_{1} & 0 \\
              E_{5} & 0 & 0 & 0 & - E_{1}   \\
           \end{array}
         \right).
\end{equation}

\par In this case, $(\phi, \xi, \eta, g)$ is a para-Sasakian structure on $M$ and hence $M(\phi, \xi, \eta, g)$ is a 5-dimensional para-Sasakian manifold.

Using (\ref{1.11}) and the above equation, one can easily obtain the following:
\begin{equation}
\nonumber \left(
  \begin{array}{ccccc}
   \tilde{\nabla}_{E_{1}}E_{1} &\tilde{\nabla}_{E_{1}}E_{2} &\tilde{\nabla}_{E_{1}}E_{3} &\tilde{\nabla}_{E_{1}}E_{4} &\tilde{\nabla}_{E_{1}}E_{5}\\
   \tilde{\nabla}_{E_{2}}E_{1} &\tilde{\nabla}_{E_{2}}E_{2} &\tilde{\nabla}_{E_{2}}E_{3} &\tilde{\nabla}_{E_{2}}E_{4} &\tilde{\nabla}_{E_{2}}E_{5} \\
   \tilde{\nabla}_{E_{3}}E_{1} &\tilde{\nabla}_{E_{3}}E_{2} &\tilde{\nabla}_{E_{3}}E_{3} &\tilde{\nabla}_{E_{3}}E_{4} &\tilde{\nabla}_{E_{3}}E_{5} \\
  \tilde{\nabla}_{E_{4}}E_{1} &\tilde{\nabla}_{E_{4}}E_{2} &\tilde{\nabla}_{E_{4}}E_{3} &\tilde{\nabla}_{E_{4}}E_{4} &\tilde{\nabla}_{E_{4}}E_{5} \\
   \tilde{\nabla}_{E_{5}}E_{1} &\tilde{\nabla}_{E_{5}}E_{2} &\tilde{\nabla}_{E_{5}}E_{3} &\tilde{\nabla}_{E_{5}}E_{4} &\tilde{\nabla}_{E_{5}}E_{5} \\
  \end{array}
\right) =\left(
           \begin{array}{ccccc}
             0 & 0 & 0 & 0 & 0 \\
             2E_{2} & -2E_{1} & 0 & 0 & 0  \\
             2E_{3} & 0 & -2E_{1} & 0 & 0 \\
              2E_{4} & 0 & 0 & -2E_{1} & 0 \\
              2E_{5} & 0 & 0 & 0 & -2E_{1}   \\
           \end{array}
         \right).
\end{equation}

\par From above results it can be easily obtain the non-zero components of Riemannian curvature and Ricci curvature tensors as
\begin{eqnarray}
\nonumber & & R(E_{1}, E_{2})E_{1} = E_{2}, \,\,\,\,\, R(E_{1}, E_{2})E_{2} = -E_{1}, \,\,\,\,\, R(E_{1}, E_{3})E_{1} = E_{3}, \,\,\,\,\, R(E_{1}, E_{3})E_{3} = -E_{1},\\
\nonumber & & R(E_{1}, E_{4})E_{1} = E_{4}, \,\,\,\,\, R(E_{1}, E_{4})E_{4} = -E_{1}, \,\,\,\,\, R(E_{1}, E_{5})E_{1} = E_{5}, \,\,\,\,\, R(E_{1}, E_{5})E_{5} = -E_{1},\\
\nonumber & & R(E_{2}, E_{3})E_{2} = E_{3}, \,\,\,\,\, R(E_{2}, E_{3})E_{3} = -E_{2}, \,\,\,\,\, R(E_{2}, E_{4})E_{2} = E_{4}, \,\,\,\,\, R(E_{2}, E_{4})E_{4} = -E_{2},\\
\nonumber & & R(E_{2}, E_{5})E_{2} = E_{5}, \,\,\,\,\, R(E_{2}, E_{5})E_{5} = -E_{2}, \,\,\,\,\, R(E_{3}, E_{4})E_{3} = E_{4}, \,\,\,\,\, R(E_{3}, E_{4})E_{4} = -E_{3},\\
\nonumber & & R(E_{3}, E_{5})E_{3} = E_{5}, \,\,\,\,\, R(E_{3}, E_{5})E_{5} = -E_{3}, \,\,\,\,\, R(E_{4}, E_{5})E_{4} = E_{5}, \,\,\,\,\, R(E_{4}, E_{5})E_{5} = -E_{4},
\end{eqnarray}
 and
\begin{eqnarray}
\nonumber & & \tilde{R}(E_{1}, E_{2})E_{1} = 2E_{2}, \,\,\,\,\, \tilde{R}(E_{1}, E_{2})E_{2} = -2E_{1}, \,\,\,\,\, \tilde{R}(E_{1}, E_{3})E_{1} = 2E_{3}, \,\,\,\,\, \tilde{R}(E_{1}, E_{3})E_{3} = -2E_{1},\\
\nonumber & & \tilde{R}(E_{1}, E_{4})E_{1} = 2E_{4}, \,\,\,\,\, \tilde{R}(E_{1}, E_{4})E_{4} = -2E_{1}, \,\,\,\,\, \tilde{R}(E_{1}, E_{5})E_{1} = 2E_{5}, \,\,\,\,\,\,\,\,
 \tilde{R}(E_{1}, E_{5})E_{5} = -2E_{1},\\
\nonumber & & \tilde{R}(E_{2}, E_{3})E_{2} = 2E_{3}, \,\,\,\,\, \tilde{R}(E_{2}, E_{3})E_{3} = -2E_{2}, \,\,\,\,\, \tilde{R}(E_{2}, E_{4})E_{2} = 2E_{4}, \,\,\,\,\, \tilde{R}(E_{2}, E_{4})E_{4} = -2E_{2},\\
\nonumber & & \tilde{R}(E_{2}, E_{5})E_{2} = 2E_{5}, \,\,\,\,\, \tilde{R}(E_{2}, E_{5})E_{5} = -2E_{2}, \,\,\,\,\, \tilde{R}(E_{3}, E_{4})E_{3} = 2E_{4}, \,\,\,\,\, \tilde{R}(E_{3}, E_{4})E_{4} = -2E_{3},\\
 & & \tilde{R}(E_{3}, E_{5})E_{3} = 2E_{5}, \, \tilde{R}(E_{3}, E_{5})E_{5} = -2E_{3}, \, \tilde{R}(E_{4}, E_{5})E_{4} = 2E_{5}, \, \tilde{R}(E_{4}, E_{5})E_{5} = -2E_{4} \label{9.1},\\
&&\tilde{S}(E_{1}, E_{1}) = \tilde{S}(E_{2}, E_{2})=\tilde{S}(E_{3}, E_{3})=\tilde{S}(E_{4}, E_{4})=\tilde{S}(E_{5}, E_{5})=-8,\label{9.3}\\
\nonumber && \tilde{S}(X, Y)=-2(5-1)g(X, Y)=-8g(X, Y),
\end{eqnarray}
where $X = a_{1}E_{1} + b_{1}E_{2} + c_{1}E_{3} +  d_{1}E_{4} + f_{1}E_{5}$ and $Y = a_{2}E_{1} + b_{2}E_{2} + c_{2}E_{3} + d_{2}E_{4} + f_{2}E_{5}$.

 From (\ref{9.1}), (\ref{9.3}) and the expression of the endomorphism one can easily substantiate, the equation (\ref{7.4}) and hence second part of the theorem \ref{thm2} (for $L_{\tilde{S}}\neq-2$).

\begin{flushleft}
Vishnuvardhana. S.V.\\
Department of Mathematics,\\
Gitam Deemed to be University,\\
Doddaballapur, Bengaluru,\\
Karnataka-561 203, INDIA\\
\pn{\tt e-mail:{\verb+svvishnuvardhana@gmail.com+}}\\
\
\\
Venkatesha\\
Department of Mathematics,\\
Kuvempu University,\\
Shankaraghatta - 577 451, Shimoga,\\
Karnataka, INDIA.\\
\pn{\tt e-mail:{\verb+vensmath@gmail.com+}}\\
\end{flushleft}
\end{document}